\documentclass[11pt]{article}
\usepackage [dvips]{graphics}
\usepackage[centertags]{amsmath}
\usepackage{amsfonts}
\usepackage{amssymb}
\usepackage{makeidx}
\usepackage{bm}
\usepackage{newlfont}
\usepackage{amsthm} 
\usepackage{stmaryrd}
\usepackage[]{titletoc}  
\usepackage{mathrsfs}
\usepackage{stmaryrd}
 \usepackage{graphicx}  
 \usepackage{xypic}
\usepackage{hyperref}
\usepackage[paperwidth=19.5cm,paperheight=27cm,top=2.5cm,bottom=2.5cm,left=2.3cm,right=2.3cm]{geometry}

\newlength{\defbaselineskip}
\newcommand{\setlinespacing}[1]%
           {\setlength{\baselineskip}{#1 \defbaselineskip}}

\theoremstyle{plain}
\newtheorem{thm}{Theorem}[section]
\newtheorem{cor}[thm]{Corollary}
\newtheorem{lem}[thm]{Lemma}
\newtheorem{prop}[thm]{Proposition}

\newtheorem{mainthm}{Theorem}

\makeatletter\@addtoreset{equation}{section} \makeatother
\begin{document}

\title {Littlewood-type theorems for Hardy spaces in infinitely many variables}
\author{Jiaqi Ni}
\date{}
\maketitle

\noindent\textbf{Abstract:} Littlewood's theorem is one of the pioneering results in random analytic functions over the open unit disk.
In this paper, we prove some analogues of this theorem for Hardy spaces in infinitely many variables.
Our results not only cover finite-variable setting, but also apply in cases of Dirichlet series.



\section{Introduction}

We begin with a well known fact: The random series $\sum_{n=1}^\infty\pm\frac{1}{n^p}$ converges almost surely if and only if $p>\frac{1}{2}$.
This suggests that randomization of series may enjoy improved regularity.
We now turn to the setting of random analytic functions.
Throughout this paper, we assume that all random variables and random vectors are defined on a probability space $(\Omega,\mathcal{F},\mathbb{P})$.
Let $\{X_n\}_{n=0}^\infty$ be a sequence of random variables.
Given an analytic function $f(z)=\sum_{n=0}^\infty a_n z^n$ over the open unit disk $\mathbb{D}$, its randomization with respect to this sequence, denoted by $\mathcal{R}f$, is defined to be the following power series:
$$(\mathcal{R}f)(z)=\sum_{n=0}^\infty a_n X_n z^n.$$
Let $H^p(\mathbb{D})$ denote the Hardy space over $\mathbb{D}$.
The following theorem obtained by Littlewood \cite{Li2} is a milestone in the study of random analytic functions.

\begin{mainthm}[Littlewood's theorem, \cite{Li2}]\label{thm:mainA}
Let $\{X_n\}_{n=0}^\infty$ be a standard Bernoulli sequence, that is, a sequence of independent, identically distributed random variables with $\mathbb{P}(X_n=1)=\mathbb{P}(X_n=-1)=\frac{1}{2}$.
Suppose that $f$ is analytic on $\mathbb{D}$.

(1) If $f\in H^2(\mathbb{D})$, then for all $1\leq p<\infty$, $\mathcal{R}f\in H^p(\mathbb{D})$ almost surely;

(2) If $f\notin H^2(\mathbb{D})$, then for all $1\leq p<\infty$, $\mathcal{R}f\notin H^p(\mathbb{D})$ almost surely.
\end{mainthm}

\noindent It is worth mentioning that this theorem remains valid for standard Steinhaus sequences \cite{Li1, PZ} and standard Gaussian $N(0,1)$ sequences \cite{Kah, PWZ}.

Over the past century, analogues of Littlewood's theorem for other familiar function spaces have received attention from many researchers.
For the case of standard Bernoulli sequences, Paley and Zygmund \cite{PZ} prove that if $\{a_n^2\log^\delta n\}_{n=1}^\infty\in\ell^1$ for some $\delta>1$, then $\mathcal{R}f$ lies in the disk algebra $A(\mathbb{D})$ almost surely, where $\ell^1$ denotes the set of all summable sequences.
However, this fails for $\delta=1$.
When $\delta=1$, Sledd \cite{Sle} shows that $\{a_n^2\log n\}_{n=1}^\infty\in\ell^1$ implies $\mathcal{R}f\in\mathrm{BMOA}$ almost surely.
A remarkable progress is obtained by Marcus and Pisier \cite{MP} in 1978, they completely characterize when $\mathcal{R}f\in H^\infty(\mathbb{D})$ almost surely via Dudley-Fernique theorem, also see \cite{Kah, MP2}.
Later, Gao \cite{Gao} gives a necessary and sufficient condition for $\mathcal{R}f$ represents a function in the Bloch space $\mathcal{B}$ almost surely.
Recently, Cheng, Fang and Liu \cite{CFL} obtain a characterization of the pairs $(p,q)$ such that $\mathcal{R}f\in A^q(\mathbb{D})$ almost surely whenever $f\in A^p(\mathbb{D})$, where $A^p(\mathbb{D})$ and $A^q(\mathbb{D})$ denote Bergman spaces over $\mathbb{D}$.
For more related works, we refer readers to \cite{ACP, Bi, Dur, SZ}.

As we mentioned previously, Littlewood's theorem (Theorem \ref{thm:mainA}) is still true for standard Gaussian $N(0,1)$ sequences.
Lately, a version of this theorem associated with not necessarily independent Gaussian processes is proved by Cheng, Fang, Guo and Liu \cite{CFGL}.
Let $(\mathcal{X},\|\cdot\|)$ be a Banach space.
For $1\leq p<\infty$, the space $L^p(\Omega,\mathcal{X})$ is defined to be the collection of random vectors $X: \Omega\rightarrow\mathcal{X}$ for which
$$\|X\|_{L^p(\Omega,\mathcal{X})}^p=\int_{\Omega}\|X\|^p\mathrm{d}\mathbb{P}<\infty.$$
We record the their main result as follows.

\begin{mainthm}[A Gaussian version of Littlewood's theorem, \cite{CFGL}]\label{thm:mainB}
Suppose $1\leq p<\infty$.
Let $\{X_n\}_{n=0}^\infty$ be a centered real Gaussian process.

(1) If the covariance matrix $\mathbf{K}=\left(\mathbb{E}(X_m X_n)\right)_{m,n\geq0}$ is bounded on $\ell^2$, the Hilbert space of all square-summable sequences, then $\mathcal{R}$ defines a bounded linear operator from $H^2(\mathbb{D})$ to $L^2(\Omega,H^p(\mathbb{D}))$. In particular, if $f\in H^2(\mathbb{D})$, then $\mathcal{R}f\in H^p(\mathbb{D})$ almost surely.

(2) If for all $f\in H^2(\mathbb{D})$, $\mathcal{R}f\in H^p(\mathbb{D})$ almost surely, then $\mathcal{R}$ defines a bounded linear operator from $H^2(\mathbb{D})$ to $L^2(\Omega,H^p(\mathbb{D}))$.
\end{mainthm}

We now turn to main objects of our paper.
Let $\mathbb{D}^\infty=\mathbb{D}\times\mathbb{D}\times\cdots$ be the cartesian product of countably infinitely many open unit disks and $\mathbb{D}_1^\infty=\mathbb{D}^\infty\cap\ell^1$ a domain in the Banach space $\ell^1$.
For every $1\leq p<\infty$, the Hardy space $\mathbf{H}_\infty^p$ over $\mathbb{D}_1^\infty$ is defined as
$$\mathbf{H}_\infty^p=\left\{F\;\text{is analytic on}\;\mathbb{D}_1^\infty: \|F\|_p^p=\sup_{0<r<1}\int_{\mathbb{T}^\infty}|F_{[r]}|^p\mathrm{d}m_\infty<\infty\right\},$$
where $F_{[r]}(w)=F(rw_1,\ldots,r^{n}w_n,\ldots)$, $w\in\mathbb{T}^\infty$, and $\mathrm{d}m_\infty=\frac{\mathrm{d}\theta_1}{2\pi}\times\frac{\mathrm{d}\theta_2}{2\pi}\times\cdots$ denotes the Haar measure of the compact group $\mathbb{T}^\infty=\mathbb{T}\times\mathbb{T}\times\cdots$, which is the cartesian product of countably infinitely many unit circles $\mathbb{T}$.
Just like finite-variable setting, one can also define Hardy spaces over the infinite torus $\mathbb{T}^\infty$.
By a polynomial we mean that it is an analytic polynomial only depending on finitely many complex variables.
The Hardy space $H^p(\mathbb{T}^\infty)$, is defined to be the closure of $\mathcal{P}_\infty$ in $L^p(\mathbb{T}^\infty)$, where $\mathcal{P}_\infty$ denotes the ring consisting of all polynomials.
In their paper, Aleman, Olsen and Saksman \cite{AOS} shows that every function $F\in\mathbf{H}_\infty^p$ has the radial limit $F^{*}(w)=\lim_{r\rightarrow1}F_{[r]}(w)$ for almost every $w\in\mathbb{T}^\infty$.
Furthermore, the map $F\mapsto F^{*}$ defines an isometric isomorphism from $\mathbf{H}_\infty^p$ onto $H^p(\mathbb{T}^\infty)$, see \cite{AOS, BBSS, CG, DG}.
For some other recent works on Hardy spaces in infinitely many variables, we refer readers to \cite{BDFMS, DGH, GY, KQSS, Ni}.

It is known that each analytic function on the polydisk $\mathbb{D}^n$ can be represented by a power series.
We now turn to the infinite-variable setting.
Let $\mathbb{N}=\{1,2,\ldots\}$ be the set of positive integers and $p_j$ the $j$-th prime number.
With each $n\in\mathbb{N}$ is associated a unique prime factorization $n=p_{1}^{\alpha_1}\cdots p_{k}^{\alpha_k}$, and set $\alpha(n)=(\alpha_1,\ldots,\alpha_k,0,\ldots)$.
For a sequence of complex numbers $\zeta=(\zeta_1,\zeta_2,\ldots)$, write $\zeta^{\alpha(n)}=\zeta_{1}^{\alpha_1}\cdots\zeta_{k}^{\alpha_k}$.
In their paper \cite{DMP}, Defant, Maestre and Prengel prove that each analytic function $F$ on $\mathbb{D}_1^\infty$ has a unique monomial expansion
$$F(\zeta)=\sum_{n=1}^\infty a_n\zeta^{\alpha(n)},\quad\zeta\in\mathbb{D}_1^\infty,$$
which converges uniformly and absolutely on compact subsets of $\mathbb{D}_1^\infty$.
Furthermore, if $F\in\mathbf{H}_\infty^p$, then the series above converges in $\mathbb{D}_2^\infty=\mathbb{D}^\infty\cap\ell^2$ \cite{BDFMS}, and hence our definition for $\mathbf{H}_\infty^p$ coincides with that in \cite{DG}.
Let $\bm{X}=\{X_n\}_{n=1}^\infty$ be a sequence of random variables.
For each analytic function $F$ on $\mathbb{D}_1^\infty$, its randomization $\mathcal{R}_{\bm{X}}F$ with respect to $\bm{X}$ is defined to be the series as follows:
\begin{equation}\label{equ101}
(\mathcal{R}_{\bm{X}}F)(\zeta)=\sum_{n=1}^\infty a_n X_n \zeta^{\alpha(n)},
\end{equation}
and we will write it simply $\mathcal{R}F$ when no confusion can arise.
Such randomization will be considered in our paper.

One of the reasons why we study analytic functions in infinitely many variables is that they are closely related to Dirichlet series, a key object in analytic number theory.
Let $1\leq p<\infty$ and $\mathcal{P}_D$ denote the set of all Dirichlet polynomials $Q(s)=\sum_{n=1}^N a_n n^{-s}$.
For each $Q\in\mathcal{P}_D$, $t\mapsto|Q(it)|^p$ is almost periodic, and hence the limit
$$\|Q\|_{p}^{p}=\lim_{T\rightarrow\infty}\frac{1}{2T}\int_{-T}^{T}|Q(it)|^p\mathrm{d}t$$
exists, see \cite{Bes}, or \cite[Theorem 1.5.6]{QQ}.
The Hardy-Dirichlet space $\mathcal{H}^p$ is defined to be the completion of $\mathcal{P}_D$ in the norm $\|\cdot\|_p$ \cite{Bay}.
From Bohr's point of view \cite{Boh}, each $Q\in\mathcal{P}_D$ can be associated with a polynomial as follows:
$$(\mathcal{B}Q)(\zeta)=\sum_{n=1}^{N}a_n\zeta^{\alpha(n)}.$$
Furthermore, it follows from Birkhoff-Oxtoby theorem \cite[Theorem 6.5.1]{QQ} that $\|Q\|_{p}^{p}=\int_{\mathbb{T}^{\infty}}|\mathcal{B}Q|^p \mathrm{d}m_\infty$.
This shows that the Bohr correspondence $\mathcal{B}: \mathcal{P}_D\rightarrow\mathcal{P}_\infty$ can be extended to an isometric isomorphism from $\mathcal{H}^p$ onto $H^p(\mathbb{T}^\infty)$, which is a natural connection between these two spaces.
Here it is worth mentioning that Konyagin, Queff\'{e}lec, Saksman and Seip \cite{KQSS} give a Littlewood-type theorem for the space of Dirichlet series in $\mathrm{BMOA}$ recently.

The purpose of this paper is to present infinite-variable versions of Theorem \ref{thm:mainA} and Theorem \ref{thm:mainB}.
We offer a unifying treatment in both finite-variable and infinite-variable setting.
The main results are Theorem \ref{thm0201}, Theorem \ref{thm0301} and Theorem \ref{thm0303}.
Although our results are formally consistent with that in one-variable setting, we overcome some essential difficulties.
For example, the convergence of series in infinitely many variables is much more complicated than that in one-variable setting.

This paper is organized as follows.
Section 2 is dedicated to Littlewood-type theorems for infinite-variable random analytic functions associated with standard sequences, including Bernoulli sequences, Steinhaus sequences and Gaussian $N(0,1)$ sequences.
We prove an infinite-variable version of Theorem \ref{thm:mainA}.
In Section 3, we concern with Gaussian versions of Littlewood's theorem in infinite-variable setting, which generalize Theorem \ref{thm:mainB}.
Moreover, in order to deal with cases of noncentered Gaussian processes, we also study coefficient multipliers of Hardy spaces in infinitely many variables.
In Section 4, we give a brief introduction of the relationship between the Hardy-Dirichlet space $\mathcal{H}^p$ and the Hardy space $\mathbf{H}_\infty^p$.
Combining the Bohr correspondence with results in previous sections, we also present some Littlewood-type theorems for Hardy-Dirichlet spaces $\mathcal{H}^p$.

\section{Littlewood-type theorems for standard random sequences}

This section will mainly deal with Littlewood-type theorems for infinite-variable random analytic functions associated with three types of standard random sequences.
Let $\bm{X}=\{X_n\}_{n=1}^\infty$ be a sequence of independent, identically distributed random variables.
We say that: (i) $\bm{X}$ is a standard Bernoulli sequence, if $\mathbb{P}(X_n=1)=\mathbb{P}(X_n=-1)=\frac{1}{2}$ for all $n$; (ii) $\bm{X}$ is a standard Steinhaus sequence, if each $X_n$ is uniformly distributed on $\mathbb{T}$; (iii) $\bm{X}$ is a standard Gaussian $N(0,1)$ sequence, if all of $X_n$ are real Gaussian variables with zero mean and unit variance.
By a standard random sequence we mean that it is either a standard Bernoulli, Steinhaus, or Gaussian $N(0,1)$ sequence, see \cite{CFL}.
In the rest of this paper, we briefly denote ``almost surely" by ``a.s.".

Recall that a complex-valued function $F$ defined on the domain $\mathbb{D}_1^\infty\subset\ell^1$ is said to be analytic, if $F$ is locally bounded, and for each $\zeta\in\mathbb{D}_1^\infty$ and $\eta\in\ell^1$, the function $F(\zeta+z\eta)$ is analytic in parameter $z$ for $\zeta+z\eta\in\mathbb{D}_1^\infty$.
One can easily check that analytic functions are continuous \cite{Din}.

Given a standard random sequence $\bm{X}$ and an analytic function
\begin{equation}\label{equ200}
F(\zeta)=\sum_{n=1}^\infty a_n\zeta^{\alpha(n)},\quad\zeta\in\mathbb{D}_1^\infty
\end{equation}
on $\mathbb{D}_1^\infty$, we have defined the randomization $\mathcal{R}F$ with respect to $\bm{X}$ in (\ref{equ101}).
Our main result of this section is stated as follows.

\begin{thm}\label{thm0201}
Let $\bm{X}$ be a standard random sequence and $F$ an analytic function on $\mathbb{D}_1^\infty$.

(1) If $F\in\mathbf{H}_\infty^2$, then for all $1\leq p<\infty$, $\mathcal{R}F\in\mathbf{H}_\infty^p$ a.s.;

(2) If $F\notin\mathbf{H}_\infty^2$, then for all $1\leq p<\infty$, $\mathcal{R}F\notin\mathbf{H}_\infty^p$ a.s..
\end{thm}

Before proving Theorem \ref{thm0201}, we need some preparations.
It seems not easy to claim whether the random series $\mathcal{R}F$ converges in $\mathbb{D}_1^\infty$.
So we begin with its dilation
\begin{equation}\label{equ202}
(\mathcal{R}_{[r]}F)(w)=\sum_{n=1}^\infty a_n X_n r^{|\alpha(n)|_{*}} w^{\alpha(n)},\quad w\in\mathbb{T}^\infty,
\end{equation}
where $0<r<1$.

Let $\overline{\mathbb{D}}^\infty=\overline{\mathbb{D}}\times\overline{\mathbb{D}}\times\cdots$ denote the cartesian product of countably infinitely many closed unit disks $\overline{\mathbb{D}}$.
Equipped with the product topology, it is compact.
The infinite polydisk algebra, denoted by $\mathbf{A}_\infty$, is defined to be the norm-closure of $\mathcal{P}_\infty$ in $C(\overline{\mathbb{D}}^\infty)$, the Banach algebra of all continuous functions on $\overline{\mathbb{D}}^\infty$.
Since $\mathbf{A}_\infty$ is of the Shilov boundary $\mathbb{T}^\infty$, one naturally identifies it with $A(\mathbb{T}^\infty)$ by the restriction map, where $A(\mathbb{T}^\infty)$ denotes the norm-closure of $\mathcal{P}_\infty$ in the Banach algebra $C(\mathbb{T}^\infty)$ of all continuous functions on $\mathbb{T}^\infty$, see \cite{CG}.
Given an analytic function $F$ on $\mathbb{D}_1^\infty$ with the monomial expansion (\ref{equ200}), for each $0<r<1$, write
\begin{equation}\label{equ201}
F_{[r]}(w)=F(rw_1,\ldots,r^{n}w_n,\ldots)=\sum_{n=1}^\infty a_n r^{|\alpha(n)|_{*}} w^{\alpha(n)},\quad w\in\mathbb{T}^\infty,
\end{equation}
where $|\alpha(n)|_{*}=\alpha_1+\cdots+k\alpha_k$.
As mentioned in Introduction, (\ref{equ200}) converges uniformly and absolutely on the compact subset $r\mathbb{T}\times\cdots r^n\mathbb{T}\times\cdots$ of $\mathbb{D}_1^\infty$.
Therefore, the series $\sum_{n=1}^\infty|a_n|r^{|\alpha(n)|_{*}}$ converges, and hence $F_{[r]}\in \mathbf{A}_\infty$.
Furthermore, if $1\leq p<\infty$ and $F\in\mathbf{H}_\infty^p$, (\ref{equ201}) is a convergent series in $\mathbf{H}_\infty^p$, and $\|F_{[r]}-F\|_p\rightarrow0$ as $r\rightarrow1$.
For more details, we refer readers to \cite{AOS, DMP, DG}.

The following proposition shows that the random function $\mathcal{R}_{[r]}F$ defined in (\ref{equ202}) lies in the infinite polydisk algebra almost surely.

\begin{prop}\label{thm0200A}
Let $\bm{X}$ be a standard random sequence and $F$ an analytic function on $\mathbb{D}_1^\infty$.
Then for each $0<r<1$, $\mathcal{R}_{[r]}F\in\mathbf{A}_\infty$ a.s..
\end{prop}

Suppose that $0<r<1$ and $F$ is analytic on $\mathbb{D}_1^\infty$ with the monomial expansion (\ref{equ200}), we have just mentioned that $\sum_{n=1}^\infty|a_n|r^{|\alpha(n)|_{*}}<\infty$.
If $\bm{X}$ is either a standard Bernoulli, or Steinhaus sequence, then $|X_n|=1$ a.s., which implies that the series in (\ref{equ202}) converges uniformly and absolutely on $\mathbb{T}^\infty$ a.s., and hence $\mathcal{R}_{[r]}F\in\mathbf{A}_\infty$ a.s..
Now it suffices to prove Proposition \ref{thm0200A} for standard Gaussian $N(0,1)$ sequences, and the following two lemmas are needed.

\begin{lem}\label{thm0201A}
Let $\sum_{n=1}^\infty x_n$ be a series whose entries are all nonnegative. Write
$$\gamma=\limsup_{n\rightarrow\infty}x_n^{\frac{1}{|\alpha(n)|_{*}}}.$$
If $\gamma<1$, then $\sum_{n=1}^\infty x_n$ converges, and if $\gamma>1$, then $\sum_{n=1}^\infty x_n$ diverges.
\end{lem}

\begin{proof}
If $\gamma<1$, choose $\beta\in(\gamma,1)$. We see that there is a positive integer $N$ for which if $n>N$, then $x_n^{\frac{1}{|\alpha(n)|_{*}}}\leq\beta<1$. Therefore
$$\sum_{n=1}^\infty x_n\leq\sum_{n=1}^\infty\beta^{|\alpha(n)|_{*}}=\prod_{j=1}^\infty\frac{1}{1-\beta^{j}}<\infty.$$
If $\gamma>1$, then $\limsup_{n\rightarrow\infty} x_n\geq1$, which implies that $\sum_{n=1}^\infty x_n$ diverges.
\end{proof}

The following lemma is a consequence of Borel-Contelli lemma, which is an analogue of \cite[Lemma 6]{CFGL}.

\begin{lem}\label{thm0202A}
Let $\{X_n\}_{n=1}^\infty$ be a sequence of Gaussian $N(0,1)$ variables, then
$$\lim_{n\rightarrow\infty}|X_n|^{\frac{1}{|\alpha(n)|_{*}}}=1\quad \mathrm{a.s.}.$$
\end{lem}

We mention the Gaussian sequence in this lemma is not necessarily independent.
To prove this lemma, we present a short piece of useful reasoning.
Let $X$ be a Gaussian $N(0,1)$ random variable and
$$\rho(x)=\frac{1}{\sqrt{2\pi}}\exp\left(-\frac{x^2}{2}\right),\quad x\in\mathbb{R}$$
its density function over the real line $\mathbb{R}$.
Then for every $x\geq1$,
$$\mathbb{P}(X\geq x)\leq x\int_{x}^\infty\rho(u)\mathrm{d}u
\leq\frac{1}{\sqrt{2\pi}}\int_{x}^\infty u\exp\left(-\frac{u^2}{2}\right)\mathrm{d}u
=\frac{1}{\sqrt{2\pi}}\exp\left(-\frac{x^2}{2}\right),$$
and hence
\begin{equation}\label{equ203}
\mathbb{P}(|X|\geq x)=2\mathbb{P}(X\geq x)\leq\sqrt{\frac{2}{\pi}}\exp\left(-\frac{x^2}{2}\right).
\end{equation}
On the other hand, for every $x>0$, it is clear that
\begin{equation}\label{equ203B}
\mathbb{P}(|X|\leq x)=2\int_{0}^x\rho(u)\mathrm{d}u=\sqrt{\frac{2}{\pi}}\int_{0}^x\exp\left(-\frac{u^2}{2}\right)\mathrm{d}u\leq\sqrt{\frac{2}{\pi}}x.
\end{equation}

\noindent\textbf{Proof of Lemma \ref{thm0202A}.} For each $\varepsilon>0$ and $n\in\mathbb{N}$, write
$$A_{n,\varepsilon}=\left\{|X_n|\geq(1+\varepsilon)^{|\alpha(n)|_{*}}\right\}.$$
Then by (\ref{equ203}),
\begin{equation}\label{equ204}
\sum_{n=1}^\infty\mathbb{P}(A_{n,\varepsilon})\leq\sqrt{\frac{2}{\pi}}\sum_{n=1}^\infty\exp\left[-\frac{(1+\varepsilon)^{2|\alpha(n)|_{*}}}{2}\right].
\end{equation}
Since for every $N>0$, the inequality $|\alpha(n)|_{*}\leq N$ has only finitely many solutions, we see that $|\alpha(n)|_{*}\rightarrow\infty$ as $n\rightarrow\infty$, and hence
$$\left\{\exp\left[-\frac{(1+\varepsilon)^{2|\alpha(n)|_{*}}}{2}\right]\right\}^{\frac{1}{|\alpha(n)|_{*}}}
=\exp\left[-\frac{(1+\varepsilon)^{2|\alpha(n)|_{*}}}{2|\alpha(n)|_{*}}\right]\rightarrow0$$
as $n\rightarrow\infty$.
Then it follows from Lemma \ref{thm0201A} that the right side of (\ref{equ204}) is finite, which reveals that for every $\varepsilon>0$, $\sum_{n=1}^\infty\mathbb{P}(A_{n,\varepsilon})<\infty$.
Hence by Borel-Contelli lemma, for every $\varepsilon>0$, $\mathbb{P}(\limsup_{n\rightarrow\infty}A_{n,\varepsilon})=0$, which gives that
\begin{equation}\label{equ205}
\limsup_{n\rightarrow\infty}|X_n|^{\frac{1}{|\alpha(n)|_{*}}}\leq1\quad \mathrm{a.s.}.
\end{equation}
On the other hand, for each $\varepsilon>0$ and $n\in\mathbb{N}$, write
$$B_{n,\varepsilon}=\left\{|X_n|\leq(1-\varepsilon)^{|\alpha(n)|_{*}}\right\}.$$
It follows from (\ref{equ203B}) that
$$\sum_{n=1}^\infty\mathbb{P}(B_{n,\varepsilon})\leq\sqrt{\frac{2}{\pi}}\sum_{n=1}^\infty(1-\varepsilon)^{|\alpha(n)|_{*}}=\sqrt{\frac{2}{\pi}}\prod_{j=1}^\infty\frac{1}{1-(1-\varepsilon)^{j}}<\infty.$$
Again by Borel-Contelli lemma, for every $\varepsilon>0$, $\mathbb{P}(\limsup_{n\rightarrow\infty}B_{n,\varepsilon})=0$, which implies that
$$\liminf_{n\rightarrow\infty}|X_n|^{\frac{1}{|\alpha(n)|_{*}}}\geq1\quad \mathrm{a.s.}.$$
Combining this with (\ref{equ205}) yields the desired conclusion. $\hfill\square$

\vskip2mm

Now we can give the proof of Proposition \ref{thm0200A}.

\vskip2mm

\noindent\textbf{Proof of Proposition \ref{thm0200A}.} Without loss of generality, assume that $\bm{X}$ is a standard Gaussian $N(0,1)$ sequence.
Suppose that $F$ has the monomial expansion (\ref{equ200}), it suffices to show that for each $0<r<1$, $\sum_{n=1}^\infty |a_n||X_n|r^{|\alpha(n)|_{*}}$ converges a.s..
Choose $r_0\in(r,1)$.
We have mentioned after Theorem \ref{thm0201} that the series $\sum_{n=1}^\infty|a_n|r_0^{|\alpha(n)|_{*}}$ converges.
Hence by Lemma \ref{thm0201A}, $\limsup_{n\rightarrow\infty}|a_n|^{\frac{1}{|\alpha(n)|_{*}}}r_0\leq1$.
Then it follows from Lemma \ref{thm0202A} that
$$\limsup_{n\rightarrow\infty}\left(|a_n||X_n|r^{|\alpha(n)|_{*}}\right)^{\frac{1}{|\alpha(n)|_{*}}}
=r\limsup_{n\rightarrow\infty}|a_n|^{\frac{1}{|\alpha(n)|_{*}}}\lim_{n\rightarrow\infty}|X_n|^{\frac{1}{|\alpha(n)|_{*}}}\leq\frac{r}{r_0}<1\quad \mathrm{a.s.}.$$
Again by Lemma \ref{thm0201A}, $\sum_{n=1}^\infty |a_n||X_n|r^{|\alpha(n)|_{*}}$ converges a.s., as desired. $\hfill\square$

\vskip2mm

The key to proving Theorem \ref{thm0201} is to discuss whether
\begin{equation}\label{equ205B}
\|\mathcal{R}F\|_p^p=\sup_{0<r<1}\int_{\mathbb{T}^\infty}|\mathcal{R}_{[r]}F|^p\mathrm{d}m_\infty<\infty\quad \mathrm{a.s.}.
\end{equation}
By Proposition \ref{thm0200A}, for each $0<r<1$, $\mathcal{R}_{[r]}F\in\mathbf{A}_\infty$ a.s., and hence integrals in (\ref{equ205B}) increase with $r$, a.s..
Fore more details, we refer readers to \cite[Section 2]{DG}.
The following proposition presents some equivalent characterizations for the finiteness of $\|\mathcal{R}F\|_p$.
For one-variable setting, see \cite{CFL}.

\begin{prop}\label{thm0203}
Suppose that $1\leq p<\infty$ and $\bm{X}$ is a standard random sequence. If $F$ is analytic on $\mathbb{D}_1^\infty$, then the following statements are equivalent:

(1) $\|\mathcal{R}F\|_p$ is finite a.s.;

(2) $\{S_n\mathcal{R}F\}_{n=1}^\infty$ is bounded in $\mathbf{H}_\infty^p$ a.s., where $S_n\mathcal{R}F$ denotes the $n$-th partial sum of $\mathcal{R}F$;

(3) There exists some $\lambda>0$ such that $\mathbb{E}(\exp(\lambda\|\mathcal{R}F\|_{p}^p))$ is finite;

(4) For every $\beta>0$, $\mathbb{E}(\|\mathcal{R}F\|_{p}^\beta)$ is finite;

(5) There exists some $\beta>0$ such that $\mathbb{E}(\|\mathcal{R}F\|_{p}^\beta)$ is finite.
\end{prop}

\begin{proof}
$(1)\Rightarrow(2)$: Write $Z_{n}=a_n X_n\zeta^{\alpha(n)}$.
It is clear that $\{Z_{n}\}_{n=1}^\infty$ defines a sequence of independent and symmetric random vectors with values in $\mathbf{H}_\infty^p$.
Choose a sequence $r_m\rightarrow1$ and put $b_{mn}=r_m^{|\alpha(n)|_{*}}$.
Then $b_{mn}\rightarrow1$ as $m\rightarrow\infty$, and
$$(\mathcal{R}_{[r_m]}F)(w)=\sum_{n=1}^\infty a_n X_n r_m^{|\alpha(n)|_{*}} w^{\alpha(n)}=\sum_{n=1}^\infty b_{mn}Z_{n}(w),\quad w\in\mathbb{T}^\infty.$$
By the proof of Proposition \ref{thm0200A}, for each $m\in\mathbb{N}$, $\sum_{n=1}^\infty b_{mn}Z_{n}$ converges to $\mathcal{R}_{[r_m]}F\in\mathbf{A}_\infty$ uniformly on $\mathbb{T}^\infty$ a.s..
Therefore, $\mathcal{R}_{[r_m]}F=\sum_{n=1}^\infty b_{mn}Z_{n}$ is also a convergent series in $\mathbf{H}_\infty^p$ a.s..
On the other hand, the finiteness of $\|\mathcal{R}F\|_p$ implies that $\{\mathcal{R}_{[r_m]}F\}_{m=1}^\infty$ is bounded in $\mathbf{H}_\infty^p$ a.s..
Therefore, by Marcinkiewitz-Zygmund-Kahane theorem \cite[Chapter 6, Themrem II.4]{LQ}, $S_n\mathcal{R}F=\sum_{k=1}^n Z_k$ is bounded in $\mathbf{H}_\infty^p$ a.s..

$(2)\Rightarrow(3)$: Write $M=\sup_{n\geq1}\|S_n\mathcal{R}F\|_{p}^p$, then $M$ is finite a.s..
Hence by Fernique's theorem \cite[Lemma 9]{CFL}, $\mathbb{E}(\exp(\lambda M))$ is finite when $\lambda>0$ is small enough.
As mentioned before, for every $0<r<1$, $\{S_n\mathcal{R}_{[r]}F\}_{n=1}^\infty$ converges to $\mathcal{R}_{[r]}F$ in $\mathbf{H}_\infty^p$-norm a.s..
Therefore,
$$\|\mathcal{R}F\|_p^p=\sup_{0<r<1}\|\mathcal{R}_{[r]}F\|_p^p=\sup_{0<r<1}\lim_{n\rightarrow\infty}\|S_n\mathcal{R}_{[r]}F\|_p^p\leq\liminf_{n\rightarrow\infty}\|S_n\mathcal{R}F\|_p^p\leq M\quad \mathrm{a.s.},$$
which implies that
$$\mathbb{E}(\exp(\lambda\|\mathcal{R}F\|_{p}^p))\leq\exp(\lambda M)<\infty.$$

$(3)\Rightarrow(4)$: For a fixed $\beta>0$, we have $x^{\frac{\beta}{p}}\exp(-x)\rightarrow0$ as $x\rightarrow+\infty$, and hence there is a constant $C>0$ such that $x^{\frac{\beta}{p}}\leq C\exp x$ for all $x\geq0$.
Put $x=\lambda\|\mathcal{R}F\|_p^p$, then
$$\|\mathcal{R}F\|_{p}^\beta\leq C\lambda^{-\frac{\beta}{p}}\exp(\lambda\|\mathcal{R}F\|_{p}^p).$$
Combining this with (3) yields that $\mathbb{E}(\|\mathcal{R}F\|_{p}^\beta)$ is finite.

$(4)\Rightarrow(5)\Rightarrow(1)$: Obviously.
\end{proof}

Now we present the proof of Theorem \ref{thm0201}.
Here we mention that an analytic function $F$ on $\mathbb{D}_1^\infty$ with monomial expansion (\ref{equ200}) lies in the Hardy space $\mathbf{H}_\infty^2$ if and only if $\{a_n\}_{n=1}^\infty\in\ell^2$, see \cite{Ni}.

In what follows, ``$A\sim B$" means that $A$ is comparable to $B$, that is, there are absolute constants $C_1, C_2>0$ such that $C_1 A\leq B\leq C_2 A$.

\vskip2mm

\noindent\textbf{Proof of Theorem \ref{thm0201}.} Since integrals in (\ref{equ205B}) increase with $r$, a.s., it follows from the monotone convergence theorem and Fubini's theorem that
\begin{equation}\label{equ206}
\begin{aligned}
\mathbb{E}\left(\|\mathcal{R}F\|_{p}^p\right)&=\int_{\Omega}\left(\sup_{0<r<1}\int_{\mathbb{T}^\infty}|\mathcal{R}_{[r]}F|^p\mathrm{d}m_\infty\right)\mathrm{d}\mathbb{P}\\
&=\sup_{0<r<1}\int_{\Omega}\int_{\mathbb{T}^\infty}|\mathcal{R}_{[r]}F|^p\mathrm{d}m_\infty\mathrm{d}\mathbb{P}\\
&=\sup_{0<r<1}\int_{\mathbb{T}^\infty}\int_{\Omega}|\mathcal{R}_{[r]}F|^p\mathrm{d}\mathbb{P}\mathrm{d}m_\infty.
\end{aligned}
\end{equation}
For a fixed $r$ and $w\in\mathbb{T}^\infty$, write $e_n(w)=a_n r^{|\alpha(n)|_{*}}w^{\alpha(n)}$.
Then $\sum_{n=1}^\infty e_n(w)X_n$ converges to $(\mathcal{R}_{[r]}F)(w)$ a.s., and the Khintchine-Kahane inequality \cite[Lemma 11]{CFL} implies that
$$\int_{\Omega}|\mathcal{R}_{[r]}F|^p\mathrm{d}\mathbb{P}\sim\left(\int_{\Omega}|\mathcal{R}_{[r]}F|^2\mathrm{d}\mathbb{P}\right)^{\frac{p}{2}}=\left(\sum_{n=1}^\infty|a_n|^2 r^{2|\alpha(n)|_{*}}\right)^{\frac{p}{2}}.$$
Combining this with (\ref{equ206}) yields that
$$\mathbb{E}\left(\|\mathcal{R}F\|_{p}^p\right)\sim\sup_{0<r<1}\left(\sum_{n=1}^\infty|a_n|^2 r^{2|\alpha(n)|_{*}}\right)^{\frac{p}{2}}=\left(\sum_{n=1}^\infty|a_n|^2\right)^{\frac{p}{2}}.$$
Therefore, $\mathbb{E}\left(\|\mathcal{R}F\|_{p}^p\right)$ is finite if and only if $F\in\mathbf{H}_\infty^2$.

(1) If $F\in\mathbf{H}_\infty^2$, it will be shown that $\mathcal{R}F$ is analytic on $\mathbb{D}_1^\infty$ and $\|\mathcal{R}F\|_p$ is finite for $1\leq p<\infty$ a.s..
Indeed, the series $\sum_{n=1}^\infty a_n z^n$ defines a function in $H^2(\mathbb{D})$.
Hence by Theorem \ref{thm:mainA}, $\sum_{n=1}^\infty a_n X_nz^n$ defines a function in $H^2(\mathbb{D})$ a.s.. Therefore, $\{a_n X_n\}_{n=1}^\infty\in\ell^2$ a.s., and thus $\mathcal{R}F$ is analytic on $\mathbb{D}_1^\infty$ a.s..
On the other hand, since $F\in\mathbf{H}_\infty^2$, $\mathbb{E}(\|\mathcal{R}F\|_{p}^p)$ is finite, which reveals that $\|\mathcal{R}F\|_p$ is finite a.s..

(2) If $F\notin\mathbf{H}_\infty^2$, then $\mathbb{E}(\|\mathcal{R}F\|_{p}^p)$ is infinite.
By Proposition \ref{thm0203}, for each $1\leq p<\infty$, $\|\mathcal{R}F\|_p$ is not a.s. finite.
Since $\|\mathcal{R}F\|_p<\infty$ is a tail event, Kolmogorov's zero-one law implies that $\|\mathcal{R}F\|_p=\infty$ a.s., and hence $\mathcal{R}F\notin\mathbf{H}_\infty^p$ a.s.. $\hfill\square$

\section{Littlewood-type theorems for Gaussian processes}

We have proved in the previous section that when $\{X_n\}_{n=1}^\infty$ is a standard Gaussian $N(0,1)$ sequence and $1\leq p<\infty$, the randomization $\mathcal{R}F$ of each $F\in\mathbf{H}_\infty^2$ lies in $\mathbf{H}_\infty^p$ a.s..
This section will discuss cases of general Gaussian processes.
A sequence of real Gaussian variables $\bm{X}=\{X_n\}_{n=1}^\infty$ is said to be a Gaussian process, if $(X_{n_1},\ldots,X_{n_{k}})$ is a Gaussian random vector for every $k$ and $n_1,\ldots,n_k\in\mathbb{N}$.
Furthermore, a Gaussian process $\bm{X}$ is said to be centered, if $\mathbb{E}X_n=0$ for all $n$.
In this section, we will no longer assume that $\{X_n\}_{n=1}^\infty$ are independent.

\subsection{Cases associated with centered Gaussian processes}

We start with centered Gaussian processes and refer readers to \cite{CFGL} for one-variable setting.
The first main result in this subsection is stated as follows, which is a generalization of Theorem \ref{thm:mainB} (1).

\begin{thm}\label{thm0301}
Let $1\leq p<\infty$ and $\bm{X}$ be a centered Gaussian process with covariance matrix $\mathbf{K}=\left(\mathbb{E}(X_m X_n)\right)_{m,n\geq1}$. If $\mathbf{K}$ is bounded on $\ell^2$, then $\mathcal{R}$ defines a bounded linear operator from $\mathbf{H}_\infty^2$ to $L^2(\Omega,\mathbf{H}_\infty^p)$. In particular, if $F\in\mathbf{H}_\infty^2$, then $\mathcal{R}F\in\mathbf{H}_\infty^p$ a.s..
\end{thm}

To prove this theorem, we recall that if $\{Z_n\}_{n=1}^\infty$ is a sequence of centered Gaussian variables which converges to $Z$ a.s., then $Z$ is also centered Gaussian \cite[pp.3]{Le}.
Also, it is worth introducing some notions of centered Gaussian vectors in Banach spaces, see \cite[Section 3.1]{LT}.
Let $\mathcal{X}$ be a separable Banach space and $\mathcal{X}^{*}$ its dual space.
We say that $X: \Omega\rightarrow\mathcal{X}$ is centered Gaussian, if for each $\varphi\in\mathcal{X}^{*}$, $L_\varphi(X)=\varphi(X)$ is centered Gaussian.
A significant result is that if $1\leq p,q<\infty$, then there is a constant $K_{p,q}>0$ such that for each centered Gaussian vector $X: \Omega\rightarrow\mathcal{X}$,
\begin{equation}\label{equ301}
\|X\|_{L^p(\Omega,\mathcal{X})}\leq K_{p,q}\|X\|_{L^q(\Omega,\mathcal{X})}.
\end{equation}
The following lemma gives some examples of centered Gaussian vectors in $L^p(\mathbb{T}^\infty)\;(1\leq p<\infty)$.

\begin{lem}\label{thm0302}
Let $\bm{X}=\{X_n\}_{n=1}^\infty$ be a centered Gaussian process.
Assume $1\leq p<\infty$ and $F$ is an analytic function on $\mathbb{D}_1^\infty$.

(1) If $\mathcal{R}F$ is analytic on $\mathbb{D}_1^\infty$ a.s., then for each $0<r<1$, $\mathcal{R}_{[r]}F$ is centered Gaussian in $L^p(\mathbb{T}^\infty)$;

(2) If $\mathcal{R}F\in \mathbf{H}_\infty^p$ a.s., then its ``boundary function"
$$(\mathcal{R}F)^{*}(w)=\lim_{r\rightarrow1}(\mathcal{R}_{[r]}F)(w),\quad w\in\mathbb{T}^\infty$$
is centered Gaussian in $L^{p}(\mathbb{T}^\infty)$.
\end{lem}

\begin{proof}
Let $q$ be the conjugate of $p$. Then $X: \Omega\rightarrow L^{p}(\mathbb{T}^\infty)$ is centered Gaussian if and only if for every $g\in L^{q}(\mathbb{T}^\infty)$,
$$L_g(X)=\int_{\mathbb{T}^\infty} Xg\mathrm{d}m_\infty$$
is centered Gaussian.

(1) For every $g\in L^{q}(\mathbb{T}^\infty)$ and $0<r<1$, write
$$Y_{r,n}=a_n X_n r^{|\alpha(n)|_{*}}\int_{\mathbb{T}^\infty}w^{\alpha(n)}g(w)\mathrm{d}m_\infty(w).$$
Then each $Y_{r,n}$ is centered Gaussian.
Since $\mathcal{R}F$ is analytic on $\mathbb{D}_1^\infty$ a.s., the series in (\ref{equ302}) below converges uniformly on $\mathbb{T}^\infty$ a.s..
Therefore, by exchanging the order of integration and summation, we see that
\begin{equation}\label{equ302}
L_g(\mathcal{R}_{[r]}F)=\int_{\mathbb{T}^\infty}\left[\sum_{n=1}^\infty a_n X_n r^{|\alpha(n)|_{*}} w^{\alpha(n)}\right]g(w)\mathrm{d}m_\infty=\sum_{n=1}^\infty Y_{r,n} \quad\mathrm{a.s.},
\end{equation}
and hence $L_g(\mathcal{R}_{[r]}F)$ is centered Gaussian.
Then $\mathcal{R}_{[r]}F$ is centered Gaussian in $L^p(\mathbb{T}^\infty)$, as desired.

(2) If $\mathcal{R}F\in \mathbf{H}_\infty^p$ a.s., then $\mathcal{R}_{[r]}F$ converges to $(\mathcal{R}F)^{*}$ in $L^p$-norm as $r\rightarrow1$ a.s., which implies that for every $g\in L^{q}(\mathbb{T}^\infty)$,
$$L_g((\mathcal{R}F)^{*})=\lim_{r\rightarrow1}L_g(\mathcal{R}_{[r]}F)\quad\mathrm{a.s.}.$$
As shown in (1), $L_g(\mathcal{R}_{[r]}F)$ is centered Gaussian. Then $L_g((\mathcal{R}F)^{*})$ is also centered Gaussian, and hence $(\mathcal{R}F)^{*}$ is centered Gaussian in $L^{p}(\mathbb{T}^\infty)$.
\end{proof}

Now we present the proof of Theorem \ref{thm0301}.

\vskip2mm

\noindent\textbf{Proof of Theorem \ref{thm0301}.} Since $\mathbf{K}$ is bounded on $\ell^2$, it follows from Theorem \ref{thm:mainB} that for each $\{a_n\}_{n=1}^\infty\in\ell^2$, $\{a_n X_n\}_{n=1}^\infty\in\ell^2$ a.s.. Therefore, for every $F\in\mathbf{H}_\infty^2$, $\mathcal{R}F$ is analytic on $\mathbb{D}_1^\infty$ a.s..
Since $\|\mathcal{R}F\|_p=\lim_{r\rightarrow1}\|\mathcal{R}_{[r]}F\|_p$, we see from Fatou's lemma that
$$\|\mathcal{R}F\|_{L^2(\Omega,\mathbf{H}_\infty^p)}\leq\liminf_{r\rightarrow1}\|\mathcal{R}_{[r]}F\|_{L^2(\Omega,\mathbf{H}_\infty^p)}.$$
As proved in Lemma \ref{thm0302}, for each $0<r<1$, $\mathcal{R}_{[r]}F$ is centered Gaussian in $L^p(\mathbb{T}^\infty)$.
Hence by (\ref{equ301}), $\|\mathcal{R}_{[r]}F\|_{L^2(\Omega,\mathbf{H}_\infty^p)}\leq K_{2,p}\|\mathcal{R}_{[r]}F\|_{L^p(\Omega,\mathbf{H}_\infty^p)}$, and thus
\begin{equation}\label{equ303}
\|\mathcal{R}F\|_{L^2(\Omega,\mathbf{H}_\infty^p)}\leq K_{2,p}\liminf_{r\rightarrow1}\|\mathcal{R}_{[r]}F\|_{L^p(\Omega,\mathbf{H}_\infty^p)}.
\end{equation}
Now we are going to give an estimate of $\|\mathcal{R}_{[r]}F\|_{L^p(\Omega,\mathbf{H}_\infty^p)}$. By Fubini's theorem,
\begin{equation}\label{equ304}
\|\mathcal{R}_{[r]}F\|_{L^p(\Omega,\mathbf{H}_\infty^p)}^p=\int_{\Omega}\int_{\mathbb{T}^\infty}|\mathcal{R}_{[r]}F|^p\mathrm{d}m_\infty\mathrm{d}\mathbb{P}
=\int_{\mathbb{T}^\infty}\int_{\Omega}|\mathcal{R}_{[r]}F|^p\mathrm{d}\mathbb{P}\mathrm{d}m_\infty.
\end{equation}
For every fixed $w\in\mathbb{T}^\infty$, the series
\begin{equation}\label{equ305}
(\mathcal{R}_{[r]}F)(w)=\sum_{n=1}^\infty a_n X_n r^{|\alpha(n)|_{*}}w^{\alpha(n)}
\end{equation}
converges, and hence it is centered Gaussian.
Then by (\ref{equ304}),
$$\|\mathcal{R}_{[r]}F\|_{L^p(\Omega,\mathbf{H}_\infty^p)}^p=\int_{\mathbb{T}^\infty}\int_{\Omega}|\mathcal{R}_{[r]}F|^p\mathrm{d}\mathbb{P}\mathrm{d}m_\infty
\leq K_{p,2}^p\int_{\mathbb{T}^\infty}\left(\int_{\Omega}|\mathcal{R}_{[r]}F|^2\mathrm{d}\mathbb{P}\right)^{\frac{p}{2}}\mathrm{d}m_\infty,$$
where the inequality follows from (\ref{equ301}).
Combining this with (\ref{equ303}) and (\ref{equ305}) yields that
\begin{equation}\label{equ306}
\|\mathcal{R}F\|_{L^2(\Omega,\mathbf{H}_\infty^p)}^p\leq C_1\liminf_{r\rightarrow1}\int_{\mathbb{T}^\infty}\left(\int_{\Omega}\left|\sum_{n=1}^\infty a_n X_n r^{|\alpha(n)|_{*}}w^{\alpha(n)}\right|^2\mathrm{d}\mathbb{P}\right)^{\frac{p}{2}}\mathrm{d}m_\infty(w),
\end{equation}
where $C_1=(K_{p,2}K_{2,p})^p$.
Since $\mathbf{K}$ is bounded on $\ell^2$, $\{X_n\}_{n=1}^\infty$ is a Bessel sequence in the Hilbert space $L^2(\Omega)$ \cite[Lemma 3.5.1]{Chr}.
Hence by \cite[Theorem 3.2.3]{Chr}, there is a constant $C_2>0$ such that for every $w\in\mathbb{T}^\infty$,
$$\int_{\Omega}\left|\sum_{n=1}^\infty a_n X_n r^{|\alpha(n)|_{*}}w^{\alpha(n)}\right|^2\mathrm{d}\mathbb{P}
\leq C_2\sum_{n=1}^\infty|a_n|^2r^{2|\alpha(n)|_{*}}.$$
Substituting this into (\ref{equ306}) yields
$$\|\mathcal{R}F\|_{L^2(\Omega,\mathbf{H}_\infty^p)}\leq C_1^{\frac{1}{p}}C_2^{\frac{1}{2}}\liminf_{r\rightarrow1}\left(\sum_{n=1}^\infty|a_n|^2r^{2|\alpha(n)|_{*}}\right)^{\frac{1}{2}}=C_1^{\frac{1}{p}}C_2^{\frac{1}{2}}\|F\|_2,$$
which completes the proof. $\hfill\square$

\vskip2mm

Assume that $1\leq p<\infty$. It is clear that if $\mathcal{R}$ defines a bounded linear operator from $\mathbf{H}_\infty^2$ to $L^2(\Omega,\mathbf{H}_\infty^p)$, then $F\in\mathbf{H}_\infty^2$ implies that $\mathcal{R}F\in\mathbf{H}_\infty^p$ a.s..
In what follows, we give its converse for centered Gaussian processes, which is the infinite-variable setting of Theorem \ref{thm:mainB} (2).

\begin{thm}\label{thm0303}
Let $1\leq p<\infty$ and $\bm{X}$ be a centered Gaussian process. If for every $F\in\mathbf{H}_\infty^2$, $\mathcal{R}F\in\mathbf{H}_\infty^p$ a.s., then $\mathcal{R}$ defines a bounded linear operator from $\mathbf{H}_\infty^2$ to $L^2(\Omega,\mathbf{H}_\infty^p)$.
\end{thm}

To prove this theorem, we need the following technical lemma.

\begin{lem}\label{thm0304}
Let $\bm{X}$ be a centered Gaussian process with $\mathrm{Var}(X_n)=\sigma_n^2$. If for every $\{a_n\}_{n=1}^\infty\in\ell^2$, the power series $\sum_{n=1}^\infty a_n X_n \zeta^{\alpha(n)}$ defines an analytic function on $\mathbb{D}_1^\infty$ a.s., then for every $0<r<1$,
$$\sum_{n=1}^\infty \sigma_n^2 r^{2|\alpha(n)|_{*}}<\infty.$$
\end{lem}

\begin{proof}
For a fixed sequence $\{a_n\}_{n=1}^\infty\in\ell^2$, we see from the analyticity of $\sum_{n=1}^\infty a_n X_n \zeta^{\alpha(n)}$ that for every $0<r_0<1$, $\sum_{n=1}^\infty a_n X_n r_0^{|\alpha(n)|_{*}}$ converges absolutely a.s..
Hence by Lemma \ref{thm0201A},
\begin{equation}\label{equ307}
\limsup_{n\rightarrow\infty}|a_n X_n|^{\frac{1}{|\alpha(n)|_{*}}}\leq\frac{1}{r_0}\quad \mathrm{a.s.}.
\end{equation}
Write $X_n=\sigma_n Y_n$, where $Y_n$ is a Gaussian $N(0,1)$ variable. Then it follows from (\ref{equ307}) and Lemma \ref{thm0202A},
\begin{equation}\label{equ308}
\limsup_{n\rightarrow\infty}|a_n \sigma_n|^{\frac{1}{|\alpha(n)|_{*}}}=\limsup_{n\rightarrow\infty}|a_n \sigma_n Y_n|^{\frac{1}{|\alpha(n)|_{*}}}=
\limsup_{n\rightarrow\infty}|a_n X_n|^{\frac{1}{|\alpha(n)|_{*}}}\leq\frac{1}{r_0}\quad \mathrm{a.s.}.
\end{equation}
Put $a_n=r_0^{|\alpha(n)|_{*}}$, then $\{a_n\}\in\ell^2$.
Hence by (\ref{equ308}),
$$r_0\limsup_{n\rightarrow\infty}|\sigma_n|^{\frac{1}{|\alpha(n)|_{*}}}=\limsup_{n\rightarrow\infty}|a_n \sigma_n|^{\frac{1}{|\alpha(n)|_{*}}}\leq\frac{1}{r_0}.$$
Letting $r_0\rightarrow1$ yields that $\limsup_{n\rightarrow\infty}|\sigma_n|^{\frac{1}{|\alpha(n)|_{*}}}\leq1$, which implies that
$$\limsup_{n\rightarrow\infty}\left(\sigma_n^2 r^{2|\alpha(n)|_{*}}\right)^{\frac{1}{|\alpha(n)|_{*}}}=r^2\limsup_{n\rightarrow\infty}|\sigma_n|^{\frac{2}{|\alpha(n)|_{*}}}\leq r^2<1,$$
and the desired conclusion follows from Lemma \ref{thm0201A}.
\end{proof}

\noindent\textbf{Proof of Theorem \ref{thm0303}.} We will use the uniform boundedness principle to complete the proof.
For $0<r<1$, it will be shown that $\mathcal{R}_{[r]}: \mathbf{H}_\infty^2\rightarrow L^2(\Omega,\mathbf{H}_\infty^p)$ is a bounded linear operator.
By the proof of Theorem \ref{thm0301}, there is a constant $C>0$ for which if $F\in\mathbf{H}_\infty^2$, then
\begin{equation}\label{equ309}
\|\mathcal{R}_{[r]}F\|_{L^2(\Omega,\mathbf{H}_\infty^p)}\leq C\left[\int_{\mathbb{T}^\infty}\left(\int_{\Omega}\left|\sum_{n=1}^\infty a_n X_n r^{|\alpha(n)|_{*}}w^{\alpha(n)}\right|^2\mathrm{d}\mathbb{P}\right)^{\frac{p}{2}}\mathrm{d}m_\infty(w)\right]^{\frac{1}{p}}.
\end{equation}
For each $w\in\mathbb{T}^\infty$, H\"{o}lder's inequality gives
$$\left|\sum_{n=1}^\infty a_n X_n r^{|\alpha(n)|_{*}}w^{\alpha(n)}\right|^2
\leq\left(\sum_{n=1}^\infty|a_n|^2\right)\left(\sum_{n=1}^\infty X_n^2 r^{2|\alpha(n)|_{*}}\right)
=\|F\|_2^2\left(\sum_{n=1}^\infty X_n^2 r^{2|\alpha(n)|_{*}}\right).$$
Substituting this into (\ref{equ309}) shows that
$$\|\mathcal{R}_{[r]}F\|_{L^2(\Omega,\mathbf{H}_\infty^p)}\leq C\left(\sum_{n=1}^\infty \sigma_n^2 r^{2|\alpha(n)|_{*}}\right)^{\frac{1}{2}}\|F\|_2.$$
Hence by Lemma \ref{thm0304}, $\mathcal{R}_{[r]}: \mathbf{H}_\infty^2\rightarrow L^2(\Omega,\mathbf{H}_\infty^p)$ is bounded.
To prove the boundedness of $\mathcal{R}$, we continue to verify that for every $F\in\mathbf{H}_\infty^2$, $\mathcal{R}F=\lim_{r\rightarrow1}\mathcal{R}_{[r]}F$ in $L^2(\Omega,\mathbf{H}_\infty^p)$.
As showed in Lemma \ref{thm0302}, $(\mathcal{R}F)^{*}$ is centered Gaussian in $L^{p}(\mathbb{T}^\infty)$.
Then it follows from \cite[Corollary 3.2]{LT} that when $\lambda>0$ is small enough, $\mathbb{E}(\exp(\lambda\|\mathcal{R}F\|_{p}^2))$ is finite, and hence $\|\mathcal{R}F\|_{p}\in L^2(\Omega)$.
Since for each $r$, $\|\mathcal{R}_{[r]}F-(\mathcal{R}F)^{*}\|_{p}\leq 2\|\mathcal{R}F\|_p$ a.s., we conclude from Fatou's lemma that
$$\begin{aligned}
\limsup_{r\rightarrow1}\|\mathcal{R}_{[r]}F-\mathcal{R}F\|_{L^2(\Omega,\mathbf{H}_\infty^p)}^2&=\limsup_{r\rightarrow1}\int_{\Omega}\|\mathcal{R}_{[r]}F-(\mathcal{R}F)^{*}\|_{p}^2\mathrm{d}\mathbb{P}\\
&\leq\int_{\Omega}\limsup_{r\rightarrow1}\|\mathcal{R}_{[r]}F-(\mathcal{R}F)^{*}\|_{p}^2\mathrm{d}\mathbb{P}\\
&=0.
\end{aligned}$$
That is, $\mathcal{R}F=\lim_{r\rightarrow1}\mathcal{R}_{[r]}F$ in $L^2(\Omega,\mathbf{H}_\infty^p)$, and the uniform boundedness principle reveals the desired conclusion. $\hfill\square$

\subsection{Coefficient multipliers and cases associated with general Gaussian processes }

In this subsection, we will give a Littlewood-type theorem for general Gaussian processes.
Beforehand, we introduce the notion of coefficient multipliers, see \cite{JVA}.

Given two linear spaces of complex sequences $\mathcal{X}$ and $\mathcal{Y}$, we say a complex sequence $\{\lambda_n\}_{n=1}^\infty$ is a coefficient multiplier of $\mathcal{X}$ into $\mathcal{Y}$, if for every $\{x_n\}_{n=1}^\infty\in\mathcal{X}$, $\{\lambda_n x_n\}_{n=1}^\infty\in\mathcal{Y}$.
The space of all coefficient multipliers of $\mathcal{X}$ into $\mathcal{Y}$ will be denoted by $(\mathcal{X},\mathcal{Y})$.
Let $1\leq p<\infty$.
As we have mentioned, each $F\in\mathbf{H}_\infty^p$ has a monomial expansion
$$F(\zeta)=\sum_{n=1}^\infty a_n \zeta^{\alpha(n)},\quad\zeta\in\mathbb{D}_1^\infty.$$
Therefore, we can regard $\mathbf{H}_\infty^p$ as a linear space of complex sequences if we need.
Hence the spaces of coefficient multipliers $(\mathbf{H}_\infty^2,\mathbf{H}_\infty^p)$ can be defined similarly.

The following proposition tells us that Littlewood-type theorems for general Gaussian processes can be reduced to cases of centered Gaussian processes.
Its proof is highly similar to that of \cite[Lemma 3]{CFGL}, and we omit it.

\begin{prop}\label{thm0305}
Let $1\leq p<\infty$ and $\bm{X}$ be a Gaussian process with $X_n\sim N(\mu_n,\sigma_n^2)$.
Then the following statements are equivalent:

(1) If $F\in\mathbf{H}_\infty^2$, then $\mathcal{R}_{\bm{X}}F\in\mathbf{H}_\infty^p$ a.s.;

(2) The mean value sequence $\{\mu_n\}_{n=1}^\infty\in(\mathbf{H}_\infty^2,\mathbf{H}_\infty^p)$, and if $F\in\mathbf{H}_\infty^2$, then $\mathcal{R}_{\bm{Y}}F\in\mathbf{H}_\infty^p$ a.s., where $\bm{Y}=\{X_n-\mu_n\}_{n=1}^\infty$.
\end{prop}

By Proposition \ref{thm0305}, it is crucial to clarify whether the mean value sequence $\{\mu_n\}_{n=1}^\infty$ is a coefficient multiplier of $\mathbf{H}_\infty^2$ into $\mathbf{H}_\infty^p$.
The rest of this subsection will characterize the space $(\mathbf{H}_\infty^2,\mathbf{H}_\infty^p)$ when $1\leq p\leq2$.
The case of $p>2$ is still open, even in one-variable setting \cite[pp.276]{JVA}.

\begin{thm}\label{thm0306}
When $1\leq p\leq2$, $(\mathbf{H}_\infty^2,\mathbf{H}_\infty^p)=\ell^\infty$.
\end{thm}

For one-variable setting of this theorem, we refer readers to \cite{JJ, JVA}.
To prove this theorem, we need to determine the solid core of $\mathbf{H}_\infty^p$ when $1\leq p\leq2$.
Recall that a linear space of complex sequences $\mathcal{X}$ is said to be a solid space, if $\{y_n\}_{n=1}^\infty\in\mathcal{X}$ whenever $\{x_n\}_{n=1}^\infty\in\mathcal{X}$ and $|y_n|\leq|x_n|$.
For every linear space of complex sequences $\mathcal{X}$, its solid core, denoted by $s(\mathcal{X})$, is defined to be the largest solid space contained in $\mathcal{X}$.
That is,
$$s(\mathcal{X})=\left\{\{\lambda_n\}_{n=1}^\infty: \{\lambda_n x_n\}_{n=1}^\infty\in\mathcal{X}\;\text{whenever}\;\{x_n\}_{n=1}^\infty\in\ell^\infty\right\},$$
see \cite[pp.129-132]{JVA}.

Before determining the solid core of $\mathbf{H}_\infty^p\;(1\leq p\leq2)$, we present the following lemma.

\begin{lem}\label{thm0307}
If $1\leq p<\infty$ and $F\in\mathbf{H}_\infty^p$ has the monomial expansion
$$F(\zeta)=\sum_{n=1}^\infty a_n \zeta^{\alpha(n)},\quad\zeta\in\mathbb{D}_1^\infty,$$
then for all $n\in\mathbb{N}$, $|a_n|\leq\|F\|_p$.
\end{lem}

\begin{proof}
Fix an $n\in\mathbb{N}$, choose $N$ large enough such that $\zeta^{\alpha(n)}$ only depends on variables $\zeta_1,\ldots,\zeta_N$.
Let
$$(A_N F)(\zeta)=F(\zeta_1,\ldots,\zeta_N,0,\ldots),\quad\zeta\in\mathbb{D}_1^\infty$$
be Bohr's $N$te Abschnitt of $F$, which can be considered as a function defined on the polydisk $\mathbb{D}^N$.
Then it follows from Cauchy's formula \cite[pp.18]{Sch} that for every $0<r<1$,
$$a_n=\frac{1}{\alpha(n)!}\left(\frac{\partial^{\alpha(n)}A_N F}{\partial^{\alpha_1}\zeta_1\cdots\partial^{\alpha_N}\zeta_N}\right)(0)=\frac{1}{i^N r^{|\beta(n)|_{*}}}\int_{\mathbb{T}^N}\frac{(A_N F)_{[r]}(w)}{w^{\beta(n)}}\mathrm{d}m_N(w),$$
where $\alpha(n)!=\alpha_1\cdots\alpha_N$, $\beta(n)=(\alpha_1+1,\ldots,\alpha_N+1,0,\ldots)$, and $\mathrm{d}m_N$ denotes the Haar measure of $N$-torus $\mathbb{T}^N$.
Hence by H\"{o}lder's inequality,
$$|a_n|\leq\frac{\|(A_N F)_{[r]}\|_1}{r^{|\beta(n)|_{*}}}\leq\frac{\|(A_N F)_{[r]}\|_p}{r^{|\beta(n)|_{*}}}\leq\frac{\|F\|_p}{r^{|\beta(n)|_{*}}}.$$
Letting $r\rightarrow1$ yields the desired lemma.
\end{proof}

The following proposition gives the solid core of $\mathbf{H}_\infty^p$ when $1\leq p\leq2$.
For one-variable setting, see \cite[Theorem 6.3.4]{JVA}.

\begin{prop}\label{thm0308}
When $1\leq p\leq2$, $s(\mathbf{H}_\infty^p)=\ell^2$.
\end{prop}

\begin{proof}
It suffices to show that $s(\mathbf{H}_\infty^p)\subset\ell^2$ since the inclusion in the other direction is trivial.
Let $\bm{X}$ be a standard Bernoulli sequence.
Then for every $\bm{\lambda}=\{\lambda_n\}_{n=1}^\infty\in s(\mathbf{H}_\infty^p)$,
$$(\mathcal{R}\bm{\lambda})(\zeta)=\sum_{n=1}^\infty \lambda_n X_n \zeta^{\alpha(n)},\quad\zeta\in\mathbb{D}_1^\infty$$
lies in $\mathbf{H}_\infty^p$ a.s..
For every $0<r<1$, write
$$(\mathcal{R}_{[r]}\bm{\lambda})(w)=\sum_{n=1}^\infty \lambda_n X_n r^{|\alpha(n)|_{*}} w^{\alpha(n)},\quad w\in\mathbb{T}^\infty.$$
As proved in Theorem \ref{thm0201},
\begin{equation}\label{equ310}
\int_{\mathbb{T}^\infty}\int_{\Omega}|\mathcal{R}_{[r]}\bm{\lambda}|^p\mathrm{d}\mathbb{P}\mathrm{d}m_\infty\sim
\left(\sum_{n=1}^\infty|\lambda_n|^2 r^{2|\alpha(n)|_{*}}\right)^{\frac{p}{2}}.
\end{equation}
On the other hand, since for almost every $t\in\Omega$, the $\ell^\infty$-norm of $\{X_n(t)\}_{n=1}^\infty$ is $1$, it follows from Lemma \ref{thm0307} and the closed graph theorem that there is a constant $C_1>0$ such that $\|\mathcal{R}\bm{\lambda}\|_p\leq C_1$.
Therefore, for every $0<r<1$,
$$\int_{\mathbb{T}^\infty}\int_{\Omega}\left|\mathcal{R}_{[r]}\bm{\lambda}\right|^p\mathrm{d}\mathbb{P}\mathrm{d}m_\infty
=\int_{\Omega}\int_{\mathbb{T}^\infty}\left|\mathcal{R}_{[r]}\bm{\lambda}\right|^p\mathrm{d}m_\infty\mathrm{d}\mathbb{P}
\leq\int_{\Omega}\|\mathcal{R}\bm{\lambda}\|_p^p\mathrm{d}\mathbb{P}\leq C_1^p.$$
Combining this with (\ref{equ310}) yields that there exists a $C>0$ such that for every $0<r<1$,
$\sum_{n=1}^\infty|\lambda_n|^2 r^{2|\alpha(n)|_{*}}\leq C$, and hence $\{\lambda_n\}_{n=1}^\infty\in\ell^2$.
\end{proof}

The proof of Theorem \ref{thm0306} is an immediate application of Proposition \ref{thm0308}.

\vskip2mm

\noindent\textbf{Proof of Theorem \ref{thm0306}.} Since $\mathbf{H}_\infty^2$ is a solid space, it follows from \cite[Lemma 12.4.1]{JVA} and Proposition \ref{thm0308} that
$$(\mathbf{H}_\infty^2,\mathbf{H}_\infty^p)=(\mathbf{H}_\infty^2,s(\mathbf{H}_\infty^p))=(\ell^2,\ell^2)=\ell^\infty,$$
as desired. $\hfill\square$

\section{Applications to Dirichlet series}

In this section, we will introduce the relationship between the Hardy-Dirichlet space $\mathcal{H}^p$ and the Hardy space $\mathbf{H}_\infty^p$.
As applications of results in the previous sections, we present Littlewood-type theorems for Hardy-Dirichlet spaces.

First, we briefly recall some elements from the theory of Dirichlet series.
A Dirichlet series is a series of the following form:
$$f(s)=\sum_{n=1}^\infty a_n n^{-s},$$
where $s$ is the complex variable.
For such a Dirichlet series and $\sigma\in\mathbb{R}$, write
$$f_{\sigma}(s)=\sum_{n=1}^\infty a_n n^{-(s+\sigma)}.$$
Let $1\leq p<\infty$, as defined in Introduction, the Hardy-Dirichlet space $\mathcal{H}^p$ is the completion of Dirichlet polynomials $\mathcal{P}_D$ in the norm
$$\|Q\|_{p}=\lim_{T\rightarrow\infty}\left(\frac{1}{2T}\int_{-T}^{T}|Q(it)|^p\mathrm{d}t\right)^{\frac{1}{p}},\quad Q\in\mathcal{P}_D.$$
It is worth mentioning that if $f\in\mathcal{H}^p$, then $f$ is a Dirichlet series, and for every $\sigma\geq0$, $f_\sigma\in\mathcal{H}^p$.
Furthermore, $\{\|f_\sigma\|_p\}_{\sigma\geq0}$ decrease with $\sigma$ and $\|f_\sigma-f\|_p\rightarrow0$ as $\sigma\rightarrow0^{+}$.
For more details, we refer readers to \cite{Bay, DGMS}.

As mentioned in Introduction, the Bohr correspondence $\mathcal{B}: \mathcal{H}^p\rightarrow\mathbf{H}_\infty^p$ defines an isometric isomorphism.
However, given an $f\in\mathcal{H}^p$, we can not give the form of $\mathcal{B}f$ from the definition, unless $f\in\mathcal{P}_D$.
The following proposition solves this problem, which can be obtained by \cite[Theorem 3.9]{BDFMS}.
Here we offer a more direct verification instead of using this theorem.

\begin{prop}\label{thm0401}
Let $1\leq p<\infty$ and $f(s)=\sum_{n=1}^\infty a_n n^{-s}$ be a Dirichlet series in $\mathcal{H}^p$. Then
$$(\mathcal{B}f)(\zeta)=\sum_{n=1}^\infty a_n \zeta^{\alpha(n)},\quad\zeta\in\mathbb{D}_1^\infty.$$
\end{prop}

\begin{proof}
For every $n\in\mathbb{N}$ and $\sigma>0$, let
$$(S_n f_\sigma)(s)=\sum_{k=1}^n a_k k^{-(s+\sigma)}$$
be the $n$-th partial sum of $f_\sigma$.
By Bohr's theorem in $\mathcal{H}^p$ \cite[Theorem 12.4]{DGMS}, for every $\sigma>0$, $\{S_n f_\sigma\}_{n=1}^\infty$ converges to $f_\sigma$ in $\mathcal{H}^p$, and hence $\{\mathcal{B}S_n f_\sigma\}_{n=1}^\infty$ converges to $\mathcal{B}f_\sigma$ in $\mathbf{H}_\infty^p$.
Since the point evaluation at each $\zeta\in\mathbb{D}_1^\infty$ is continuous in $\mathbf{H}_\infty^p$ \cite[Theorem 8.1]{CG},
\begin{equation}\label{equ401}
(\mathcal{B}f_\sigma)(\zeta)=\lim_{n\rightarrow\infty}(\mathcal{B}S_n f_\sigma)(\zeta)=\sum_{n=1}^\infty\frac{a_n}{n^\sigma}\zeta^{\alpha(n)},\quad\zeta\in\mathbb{D}_1^\infty.
\end{equation}
Recall that $\|f_\sigma-f\|_p\rightarrow0$ as $\sigma\rightarrow0^{+}$, we see $\|\mathcal{B}f_\sigma-\mathcal{B}f\|_p\rightarrow0$ as $\sigma\rightarrow0^{+}$.
Applying continuity of the point evaluation again shows that
$$(\mathcal{B}f)(\zeta)=\lim_{\sigma\rightarrow0^{+}}(\mathcal{B}f_\sigma)(\zeta)=\lim_{\sigma\rightarrow0^{+}}\sum_{n=1}^\infty\frac{a_n}{n^\sigma}\zeta^{\alpha(n)},\quad\zeta\in\mathbb{D}_1^\infty,$$
where the last equality follows from (\ref{equ401}).
Let $\mathbb{D}_0^\infty$ denote the set of all elements in $\mathbb{D}^\infty$ with finitely many nonzero entries and $p_n$ denote the $n$-th prime number.
For every $\zeta\in\mathbb{D}_0^\infty$, there exists $\sigma>0$ such that $(\zeta_1 p_1^{\sigma},\ldots,\zeta_n p_n^\sigma,\ldots)\in\mathbb{D}_0^\infty$.
Since the series (\ref{equ401}) converges absolutely in $\mathbb{D}_1^\infty$,
$$\sum_{n=1}^\infty a_n\zeta^{\alpha(n)}=\sum_{n=1}^\infty\frac{a_n}{n^\sigma}(\zeta_1 p_1^{\sigma},\ldots,\zeta_n p_n^\sigma,\ldots)^{\alpha(n)}$$
converges absolutely in $\mathbb{D}_0^\infty$.
Now it follows from Lebesgue's dominated convergence theorem that
\begin{equation}\label{equ402}
(\mathcal{B}f)(\zeta)=\lim_{\sigma\rightarrow0^{+}}\sum_{n=1}^\infty\frac{a_n}{n^\sigma}\zeta^{\alpha(n)}=\sum_{n=1}^\infty a_n\zeta^{\alpha(n)},\quad\zeta\in\mathbb{D}_0^\infty.
\end{equation}
For every $k\in\mathbb{N}$, let $\Xi_k$ be the multiplicative subsemigroup of $\mathbb{N}$ generated by $p_1,\ldots,p_k$.
Then by (\ref{equ402}), Bohr's $k$te Abschnitt of $\mathcal{B}f$,
$$(A_k\mathcal{B}f)(\zeta)=(\mathcal{B}f)(\zeta,0,\ldots)=\sum_{n\in\Xi_k} a_n\zeta^{\alpha(n)},\quad\zeta\in\mathbb{D}^k.$$
Therefore, the monomial expansion of $\mathcal{B}f$ on $\mathbb{D}_1^\infty$ is of the coefficient sequence $\{a_n\}_{n=1}^\infty$, and the proof is complete.
\end{proof}

For each Dirichlet series $f(s)=\sum_{n=1}^\infty a_n n^{-s}$, write $(\mathcal{B}^{*}f)(\zeta)=\sum_{n=1}^\infty a_n\zeta^{\alpha(n)}$ for its formal Bohr correspondence.
Proposition \ref{thm0401} implies that if $f\in\mathcal{H}^p\;(1\leq p<\infty)$, then $\mathcal{B}^{*}f=\mathcal{B}f$ on $\mathbb{D}_1^\infty$.
The following corollary immediately follows.

\begin{cor}\label{thm0402}
Let $1\leq p<\infty$ and $F(\zeta)=\sum_{n=1}^\infty a_n \zeta^{\alpha(n)}$ be a function in $\mathbf{H}_\infty^p$.
Then
$$(\mathcal{B}^{-1}F)(s)=\sum_{n=1}^\infty a_n n^{-s},\quad s\in\mathbb{C}_{\frac{1}{2}},$$
where $\mathbb{C}_{\frac{1}{2}}$ denotes the half plane $\{z\in\mathbb{C}: \mathrm{Re}\;z>\frac{1}{2}\}$.
\end{cor}

Let $f(s)=\sum_{n=1}^\infty a_n n^{-s}$ be a Dirichlet series and $\bm{X}=\{X_n\}_{n=1}^\infty$ a sequence of random variables.
The randomization $\mathcal{R}_{\bm{X}}f$ of $f$ with respect to $\bm{X}$ is defined to be the series as follows:
$$(\mathcal{R}_{\bm{X}}f)(s)=\sum_{n=1}^\infty a_n X_n n^{-s},$$
and we will write it simply $\mathcal{R}f$ when no confusion can arise.
Now we can present Littlewood-type theorems for Hardy-Dirichlet spaces via Proposition \ref{thm0401}, Corollary \ref{thm0402} and Littlewood-type theorems for Hardy spaces in infinitely many variables.

\begin{thm}\label{thm0403}
Let $\bm{X}$ be a standard random sequence and $f$ a Dirichlet series with formal Bohr correspondence $\mathcal{B}^{*}f$ analytic on $\mathbb{D}_1^\infty$.

(1) If $f\in\mathcal{H}^2$, then for all $1\leq p<\infty$, $\mathcal{R}f\in\mathcal{H}^p$ a.s.;

(2) If $f\notin\mathcal{H}^2$, then for all $1\leq p<\infty$, $\mathcal{R}f\notin \mathcal{H}^p$ a.s..
\end{thm}

\begin{proof}
(1) If $f\in \mathcal{H}^2$, then $\mathcal{B}f\in\mathbf{H}_\infty^2$.
Hence by Theorem \ref{thm0201}, $\mathcal{RB}f\in \mathbf{H}_\infty^p$ a.s..
Now it follows from Corollary \ref{thm0402} that $\mathcal{R}f=\mathcal{B}^{-1}\mathcal{RB}f\in\mathcal{H}^p$ a.s..

(2) By Proposition \ref{thm0401} and Corollary \ref{thm0402}, if $\mathcal{R}f\in \mathcal{H}^p$, then $\mathcal{RB}^{*}f=\mathcal{BR}f\in\mathbf{H}_\infty^p$.
Therefore, if ``$\mathcal{R}f\notin \mathcal{H}^p$ a.s." is false, we conclude from Theorem \ref{thm0201} that $\mathcal{B}^{*}f\in\mathbf{H}_\infty^2$, which implies that $f\in\mathcal{H}^2$, a contradiction.
\end{proof}

\begin{thm}
Let $1\leq p<\infty$ and $\bm{X}$ be a centered Gaussian process.

(1) If the covariance matrix $\mathbf{K}$ of $\bm{X}$ is bounded on $\ell^2$, then $\mathcal{R}$ defines a bounded linear operator from $\mathcal{H}^2$ to $L^2(\Omega,\mathcal{H}^p)$.
In particular, if $f\in\mathcal{H}^2$, then $\mathcal{R}f\in\mathcal{H}^p$ a.s..

(2) If for all $f\in\mathcal{H}^2$, $\mathcal{R}f\in\mathcal{H}^p$ a.s., then $\mathcal{R}$ defines a bounded linear operator from $\mathcal{H}^2$ to $L^2(\Omega,\mathcal{H}^p)$.
\end{thm}

\begin{proof}
(1) If $f\in\mathcal{H}^2$, then $\mathcal{B}f\in\mathbf{H}_\infty^2$.
Hence by Theorem \ref{thm0301}, there is a constant $C>0$, does not depend on $f$, such that
\begin{equation}\label{equ403}
\|\mathcal{RB}f\|_{L^2(\Omega,\mathbf{H}_\infty^p)}\leq C\|\mathcal{B}f\|_2=C\|f\|_2.
\end{equation}
In particular, $\mathcal{RB}f\in\mathbf{H}_\infty^p$ a.s., which implies that $\mathcal{R}f\in\mathcal{H}^p$ a.s., and furthermore,
$$\|\mathcal{RB}f\|_{L^2(\Omega,\mathbf{H}_\infty^p)}=\left(\int_{\Omega}\|\mathcal{RB}f\|_p^2\mathrm{d}\mathbb{P}\right)^{\frac{1}{2}}
=\left(\int_{\Omega}\|\mathcal{R}f\|_p^2\mathrm{d}\mathbb{P}\right)^{\frac{1}{2}}=\|\mathcal{R}f\|_{L^2(\Omega,\mathcal{H}^p)}.$$
Combining this with (\ref{equ403}) shows that $\|\mathcal{R}f\|_{L^2(\Omega,\mathcal{H}^p)}\leq C\|f\|_2$, which completes the proof.

(2) For every $F\in \mathbf{H}_\infty^2$, we see that $\mathcal{B}^{-1}F\in\mathcal{H}^2$, and thus $\mathcal{R}\mathcal{B}^{-1}F\in\mathcal{H}^p$ a.s..
Hence by Proposition \ref{thm0401}, $\mathcal{R}F=\mathcal{BR}\mathcal{B}^{-1}F\in \mathbf{H}_\infty^p$ a.s..
Then it follows from Theorem \ref{thm0303} that there is a constant $C>0$ such that for every $f\in \mathcal{H}^2$,
$$\|\mathcal{RB}f\|_{L^2(\Omega,\mathbf{H}_\infty^p)}\leq C\|\mathcal{B}f\|_2=C\|f\|_2.$$
A same argument as in (1) yields that $\|\mathcal{R}f\|_{L^2(\Omega,\mathcal{H}^p)}\leq C\|f\|_2$, as desired.
\end{proof}

Since each $f\in\mathcal{H}^p$ is a Dirichlet series, $\mathcal{H}^p$ can be regarded as a linear space of complex sequences.
Hence the space of coefficient multipliers $(\mathcal{H}^2, \mathcal{H}^p)$ can be defined as in the previous section.

\begin{prop}
Let $1\leq p<\infty$ and $\bm{X}$ be a Gaussian process with $X_n\sim N(\mu_n,\sigma_n^2)$.
Then the following statements are equivalent:

(1) If $f\in\mathcal{H}^2$, then $\mathcal{R}_{\bm{X}}f\in\mathcal{H}^p$ a.s.;

(2) The mean value sequence $\{\mu_n\}_{n=1}^\infty\in(\mathcal{H}^2,\mathcal{H}^p)$, and if $f\in\mathcal{H}^2$, then $\mathcal{R}_{\bm{Y}}f\in\mathcal{H}^p$ a.s., where $\bm{Y}=\{X_n-\mu_n\}_{n=1}^\infty$.
\end{prop}

\vskip2mm

\noindent\textbf{Acknowledgements:}\quad This work was partially supported by the Fundamental Research Funds for the Central Universities (2412023QD002).

\vskip3mm \noindent{Jiaqi Ni, School of Mathematics and Statistics, Northeast Normal University,
Changchun, 130024, China, E-mail: nijq849@nenu.edu.cn}

\end{document}